\documentclass{amsart}
\usepackage{amssymb,amsmath,amsthm,mathrsfs,framed}

\usepackage{tikz-cd}
\usepackage{tikz}
\usetikzlibrary{cd}
\usepackage{graphicx}
\usepackage{epstopdf}
\usepackage{comment}
\usepackage{enumerate}

\def\NZQ{\mathbb}

\def\KK{{\NZQ K}}

\def\PP{{\NZQ P}}

\def\TT{{\NZQ T}}

\newcommand{\A}{\mathcal{A}}
\newcommand{\B}{\mathcal{B}}

\newcommand{\C}{\mathbb{C}}

\newcommand{\K}{\mathbb{K}}
\newcommand{\Q}{\mathbb{Q}}
\newcommand{\R}{\mathbb{R}}
\newcommand{\T}{\mathbb{T}}

\newcommand{\V}{\mathcal{V}}
\newcommand{\Z}{\mathbb{Z}}

\newcommand{\rfl}{\mathrm{Rfl}}
\newcommand{\cf}{\textrm{cf.}}

\newcommand{\Matrix}[9]{\begin{pmatrix}#1&#4&#7\\#2&#5&#8\\#3&#6&#9\end{pmatrix}}
\newcommand{\abs}[1]{\left\lvert#1\right\rvert} 

\theoremstyle{definition}
\newtheorem{thm}{Theorem}

\newtheorem{prop}[thm]{Proposition}

\newtheorem{coro}[thm]{Corollary}

\newtheorem{exam}[thm]{Example}

\newtheorem{rema}[thm]{Remark} 

\usepackage[normalem]{ulem}

\title{Non-very generic arrangements in low dimension}
\author{Takuya Saito}
\author{Simona Settepanella}

\address[Saito]{Department of Mathematics, Hokkaido University, Japan.}
\address[Saito, Settepanella]{Department of Statistics and Economics, Torino University, Italy}
\email[Saito]{saito.takuya.p6@elms.hokudai.ac.jp}
\email[Settepanella]{simona.settepanella@unito.it}

\makeatletter
\@namedef{subjclassname@2020}{%
  \textup{2020} Mathematics Subject Classification}
\makeatother

\thanks{The named first author was supported by JSPS KAKENHI Grant Number JP23KJ0031.}

\begin{document}

 \maketitle

\begin{abstract}
The discriminantal arrangement $\B(n,k,\A)$ has been introduced by Manin and Schectman in $1989$ and it consists of all non-generic translates of a generic arrangement $\A$ of $n$ hyperplanes in a $k$-dimensional space. It is known that its combinatorics depends on the original arrangement $\A$ which, following Bayer and Brandt \cite{BB97}, is called \textit{very generic} if the intersection lattice of the induced discriminantal arrangement has maximum cardinality, \textit{non-very generic} otherwise. While a complete description of the combinatorics of $\B(n,k,\A)$ when $\A$ is very generic is known (see \cite{Ath99}), very few is known in the non-very generic case. Even to provide examples of non very generic arrangements proved to be a non-trivial task (see \cite{SSc}). 
In this paper, we characterize, classify and provide examples of non-very generic arrangements in low dimension.
\end{abstract}

\section{Introduction}
\noindent
The discriminantal arrangement $\B(n,k,\A)$, $\A$ a generic arrangement of $n$ hyperplanes in a $k$-dimensional space, has been introduced by Manin and Schectman (see \cite{MS89}) as a generalization of the well known braid arrangement with which it coincides when $k=1$. Analogously to the braid arrangement which can be regarded as the complement of the configuration space of $n$ distinct points in a line, its higher dimensional generalization, the discriminantal arrangement, can be defined as the complement of the configuration space of $n$ hyperplanes in a $k$-dimensional space such that any $k$ of them are in general position\footnote{In this case, the generalization of any two points be distinct in the line is $k$ by $k$ hyperplanes be in general position in the $k$-space. Different generalizations are also possible.}.\\
The combinatorics of the discriminantal arrangement $\B(n,k,\A)$ depends on the original arrangement $\A$ if $k > 1$ (see, among others, \cite{Fa94}) and the arrangement $\A$ is called (see \cite{BB97}) \textit{very generic} if the intersection lattice of $\B(n,k,\A)$ has maximum cardinality, \textit{non-very generic} otherwise. The space of very generic arrangements is an open Zariski set $\mathcal{Z}$ (see, among others, \cite{MS89}), but, beside this, very few is known on its characterization and to establish whether an arrangement $\A$ belongs or not to $\mathcal{Z}$ proved to be a quite difficult task (see \cite{SSc} for more details on this).\\ 
It is worthy to mention that in 1985, that is few years before Manin and Schectman, Crapo already defined an object equivalent to the discriminantal arrangement that he called \textit{geometry of circuits} (see \cite{Cr85}). The first reference to Crapo's work in the literature on the discriminantal arrangement is due to Athanasiadis (see \cite{Ath99}). In his paper Crapo presented an example of an arrangement of $6$ hyperplanes in the real plane which is, in fact, the first example of a non-very generic arrangement.
The first half of the results presented in this paper are based on this Crapo's non-very generic example (for a preliminary discussion about it see also \cite{DasPaSe21} ).
More recent results on non-very generic arrangements are in \cite{LiSe16},\cite{SaSeYa17},\cite{SaSeYa19},\cite{SeYa21},\cite{DasPaSe21}. In particular the second part of this paper completes the discussion on the case $k=3,n=6$ started by Falk in \cite{Fa94} and continued in \cite{SaSeYa19}. 
Other studies related to the combinatorics of the discriminantal arrangement include higher Bruhat orders, fiber zonotopes (cf. \cite{FeZi01}), arrangements generated by points in general position (cf. \cite{KNT12}\cite{NuTa12}), circuits of representable matroids (cf. \cite{CFW21}\cite{OxWa19}) and, more recently, applications to physics and statistics (cf. \cite{Stumal}).\\
The content of this paper is as follows. Section \ref{Sec:Pre} contains the preliminaries on the discriminantal arrangement and the definition of non-very generic intersections. In Section \ref{Sec:LowNVGP} we characterize the rank $3$ and $4$ non-very generic intersections of $\B(n,2,\A)$ by means of the Ceva's Theorem and the involutions on a projective line.
In Section \ref{Sec:Per} we classify the non-very generic intersections in $\B(6,2,\A)$ and $\B(6,3,\A)$ and provide a complete classification for the combinatorics of $\B(6,3,\A)$ over a characteristic $0$ commutative field.
Finally in Section \ref{Sec:Dodeca} we provide an example of a real arrangement $\A$ such that $\B(6,3,\A)$ admits $10$ non-very generic intersections in rank $3$.

\section*{Acknowledgment}
This work was supported by JSPS KAKENHI Grant Number JP23KJ0031.

\section{Preliminaries}\label{Sec:Pre}

\subsection{Discriminantal arrangement}Let $\A^0 = \{ H_1^0, \dots, H_n^0 \}$ be a central arrangement in $\K^k$ ($\K$ is a commutative field), $k<n$ such that any $m$ hyperplanes intersect in codimension $m$ at any point except for the origin for any $m \leq k$. We will call such an arrangement a generic arrangement. The space $\mathbb{S}[\A^0]$ (or simply $\mathbb{S}$ when dependence on $\A^0$ is clear or not essential) will denote the space of parallel translates of $\A^0$, that is the space of the arrangements $\A^t= \{ H_1^{x_1}, \dots, H_n^{x_n} \}$, $t = (x_1, \dots, x_n) \in \KK^n$, $H_i^{x_i} = H_i^0 + \alpha_i x_i$, $\alpha_i$ a vector normal to $H_i^0$. There is a natural identification of $\mathbb{S}$ with the $n$-dimensional affine space $\KK^n$ such that the arrangement $\A^0$ corresponds to the origin. In particular, an ordering of hyperplanes in $\A^0$ determines the coordinate system in $\mathbb{S}$ (see \cite{LiSe16}). \\
The closed subset of $\mathbb{S}$ formed by the translates of $\A^0$ which fail to form a generic arrangement is a union of hyperplanes $D_L \subset \mathbb{S}$ (see \cite{MS89}). Each hyperplane $D_L$ corresponds to a subset $L = \{ i_1, \dots, i_{k+1} \} \subset [n]=\{ 1, \dots, n \}$ and it consists of $n$-tuples of translates of hyperplanes $H_1^0, \dots, H_n^0$ in which translates of $H_{i_1}^0, \dots, H_{i_{k+1}}^0$ fail to form a generic arrangement. The arrangement $\B(n, k, \A)$ of hyperplanes $D_L$ is called $discriminantal$ $arrangement$ and has been introduced by Manin and Schechtman in \cite{MS89}.\\
It is well known (see, among others \cite{Cr85},\cite{MS89}) that there exists an open Zariski set $\mathcal{Z}$ in the space of generic arrangements of $n$ hyperplanes in $\KK^k$, such that the intersection lattice of the discriminantal arrangement $\mathcal{B}(n,k,\A)$ is independent from the choice of the arrangement $\A \in  \mathcal{Z}$. Accordingly to Bayer and Brandt (see \cite{BB97}) we will call the arrangements $\A \in  \mathcal{Z}$ \textit{very generic} and \textit{non-very generic} the others. 

\subsection{Non-very generic intersections}\label{subse:non-verygint} According to \cite{Ath99} if $\A$ is a very generic arrangement, then the intersection lattice of the discriminantal arrangement $\mathcal{B}(n,k,\A)$ is isomorphic to the collection of all sets $\{S_1, \ldots, S_m\}$, $S_i$ $\subset$ $[n]=\{1,\ldots,n\}$, $\abs{S_i} \geq k+1$, such that
\begin{equation}\label{eq:vgcon}
\abs{\bigcup_{i \in I} S_i} > k + \sum_{i \in I}(\abs{S_i}- k) \mbox{ for all } I \subset [m]=\{1,\ldots,m\}, \mid I \mid \geq 2 \quad .
\end{equation}
 The isomorphism is the natural one which associate to the set $S_i$ the space $D_{S_i}=\bigcap_{L \subset S_i, \mid L \mid=k+1} D_L, D_L \in \mathcal{B}(n,k,\A)$ of all translated  of $\A$ having hyperplanes indexed in $S_i$ intersecting in a not empty space. In particular $\{S_1, \ldots, S_m\}$ will correspond to the intersection $\bigcap_{i=1}^m D_{S_i}$. \\
The Athanasiadis's condition is necessary but not sufficient for an arrangement to be very generic (see \cite{SeYa21}), hence we will call \textit{non-very generic} any intersection $X$ of hyperplanes in $\mathcal{B}(n,k,\A), \A$ non-very generic arrangement, which is not combinatorially isomorphic to an intersection in $\mathcal{B}(n,k,\A'), \A'$ very generic. In particular any intersection $X=\bigcap_{i=1}^m D_{S_i}$ such that the set $\{S_1, \ldots, S_m\}$ does not satisfy the condition (\ref{eq:vgcon}), is \textit{non-very generic}.\\
In this paper we are particularly interested in the intersections of the form: 
$$X=\bigcap_{i=1}^r D_{L_i}, |L_i| = k+1 \mbox{ and } \bigcap_{i \in I}D_{L_i} \neq D_S, \mid S \mid > k+1 \mbox{ for any } I \subset [r], \mid I \mid \geq 2\quad,$$ which we will call \textit{simple} accordingly to \cite{SSc}. If we call \textit{multiplicity} of the simple intersection $X$ the number $r$ of the hyperplanes intersecting in $X$, an immediate consequence of the equation (\ref{eq:vgcon}) is the following result (see also \cite{SSc}).
\begin{prop}\label{pro:main}A simple intersection of rank strictly less than its multiplicity is non-very generic.
\end{prop}

\subsection{$\mathbf{K_{\TT}}$-translated}\label{sub:KT}
Fixed a set $\TT = \{ L_1, \dots, L_r \}$ of subsets $L_i \subset [n]$ of cardinality $k+1$ and an arrangement $\A=\{H_1,\ldots, H_n\}$ translated of $\A^0$, the intersection $P_i = \bigcap_{p \in L_i} H_p$ is a point if and only if $\A \in D_{L_i}$, it is empty otherwise.
Following \cite{SSc} we will call the set $\TT$ an $r$-\textbf{set}\footnote{{Notice that the original definition required the additional condition $L_i \cap L_j \neq \emptyset$ which we removed.}} if
 \begin{equation*}\label{eq:proper1}
 \bigcup_{i=1}^r L_i = \bigcup_{i \in I} L_i
 \end{equation*}
for any subset $I \subset [r], |I| =r-1$ and any two indices $1 \leq i < j \leq r$.
Given an $r$-set $\TT= \{ L_1, \dots, L_r \}$, a translated arrangement $\A$ of $\A^0$ will be called a $\mathbf{K_{\TT}}$-\textbf{translated} if $P_i$ is a point which belongs exactly to the $k+1$ hyperplanes indexed in $L_i$ for any $i=1,\ldots ,r$.
We will denote by $\A^{t(\T)}$ such a translate.

\section{Low rank non-very generic intersections in $\B(n,2,\A)$}\label{Sec:LowNVGP}
\noindent

In this section we will provide an algebraic way to fully characterize the non-very generic intersections in rank $3$ and $4$ in $\B(n,2,\A)$ by means of the Ceva's Theorem.

\subsection{Non-very generic intersections in rank $3$. } In \cite{Cr85} Crapo proved that an arrangement $\A$ of $6$ lines in the real plane is non-very generic if and only if it admits a translated which is combinatorially equivalent to the arrangement depicted in Figure \ref{fig;Ceva1}. We will call  such a configuration of lines \textit{Crapo's configuration}.\\
In other terms, the Crapo's configuration is a $K_{\TT}$-translated $\A^{t(\T)}$ of $\A$ such that $\A^{t(\T)}$ belongs to the simple intersection $X=\bigcap_{i=1}^4 D_{L_i}, L_i \in \T$ of multiplicity $4$ in rank $3$. In this case the rank of $X$ can be easily obtained since the only element in rank $4$ in the intersection lattice of $\B(6,2,\A)$ is the intersection $D_{[6]}$ which elements correspond to all the translated of $\A$ which are central arrangements.  By Proposition \ref{pro:main}. $X$ is a non-very generic intersection. We will call \textit{quadral point} any simple intersection of rank $3$ and multiplicity $4$.\\
In all this subsection the set $\TT$ will always denote a $4$-set of the form $\TT=\{L_1,\ldots ,L_4\}$ with $L_1=\{p_1,p_2,p_3\},L_2=\{p_1,p_5,p_6\},L_3=\{p_2,p_4,p_6\},L_4=\{p_3,p_4,p_5\}$ fixed,
unless differently specified.\\
Finally, since the discriminantal arrangement only depends on its trace at infinity $\A_\infty$, we will consider indifferently either the generic arrangement $\A$ in $\K^2$ or its trace at infinity $\A_\infty$ in $\PP(\K^2)$ .\\

\begin{figure}
\begin{tabular}{cc}
\begin{minipage}{0.45\linewidth}
\centering
	\begin{tikzpicture}[scale=2.5,  xscale=0.6, rotate=50, xscale=0.8 ]
		\draw (-0.2,0) node[at={(-0.4,0)}] {$H_{p_1}$}--(1.7,0);
		\draw (0,-0.2) node[at={(0,-0.4)}] {$H_{p_6}$}--(0,1.7);
		\draw (1.7,-0.2) node[at={(1.8,-0.4)}] {$H_{p_2}$}--(-0.2,1.7);
		\draw (1.7,-0.1) node[at={(-0.8,1.2)}] {$H_{p_3}$}--(-0.7,1.1);
		\draw (-0.7,0.8) node[at={(-0.8,0.8)}] {$H_{p_4}$}--(0.2,1.7);
		\draw (0.1,-0.2) node[at={(0.3,-0.3)}] {$H_{p_5}$}--(-0.7,1.4);
	\end{tikzpicture}
\caption{The Crapo configulation}\label{fig;Ceva1} 
\end{minipage}&
\begin{minipage}{0.45\linewidth}
\centering
	\begin{tikzpicture}[scale=2.5, rotate=90, xscale=0.6, rotate=50, xscale=0.8 ]
		\node[inner sep=0.1em, fill=black!100, circle] (a) at (0,0){};
		\node[inner sep=0.1em, fill=black!100, circle] (b) at (1.5,0){};
		\node[inner sep=0.1em, fill=black!100, circle] (c) at (0,1.5){};
		\node[inner sep=0.1em, fill=black!100, circle] (d) at (-0.5,1){};
		
		\draw (-0.2,0) --(1.7,0);
		\draw (0,-0.2) --(0,1.7);
		\draw (1.7,-0.2) --(-0.2,1.7);
		\draw (1.7,-0.1) --(-0.7,1.1);
		\draw (-0.7,0.8) --(0.2,1.7);
		\draw (0.1,-0.2) --(-0.7,1.4);

		\draw[line width=0.8pt, arrows=-stealth] (b) to node[auto=left] {$c_1\alpha_{p_1}$} (a);
		\draw[line width=0.8pt, arrows=stealth-] (c) to node[auto=left] {$c_2\alpha_{p_2}$} (b);
		\draw[line width=0.8pt, arrows=-stealth] (d) to node[auto=left] {$c_4\alpha_{p_4}$} (c);
		\draw[line width=0.8pt, arrows=-stealth] (a) to node[auto=left] {$c_5\alpha_{p_5}$} (d);
		\draw[line width=0.8pt, arrows=stealth-] (a) to node[auto=left', pos=0.3] {$c_6\alpha_{p_6}$} (c);
		\draw[line width=0.8pt, arrows=stealth-] (d) to node[auto=left'] {$c_3\alpha_{p_3}$} (b);
	\end{tikzpicture}
\caption{The Crapo configulation spanned by normals}\label{fig:Ceva2} 
\end{minipage}
\end{tabular}
\end{figure}
\noindent
The following theorem shows that the Ceva's Theorem provides a full characterization of the Crapo's configuration which can also be characterized by means of involutions,  i.e. projective transformations $f$ which satisfy $f^2=\mathrm{id}$, on the infinity line. From now on, given a family of vectors $v_i$'s, we will denote by $[v_1,v_2;v_3,v_4]$ the cross ratio $|v_1 v_3| |v_2 v_4|/|v_2 v_3| |v_1 v_4|$, where the symbol $|v_i v_j|$ stands for the determinant of the $2\times 2$ matrix consisting of the column vectors $v_i,v_j$.

\begin{thm}[The Ceva's Theorem]\label{thm;Ceva} Let $\TT$ be a $4$-set, the following statements are equivalent:
\begin{enumerate}
\item[i)] the arrangement $\A$ of $6$ lines in $\K^2$ admits a $K_{\TT}$-translated $\A^{t(\T)}$;
\item[ii)] the Ceva's equation 
\begin{equation}\label{eq;Ceva}
	\frac{|\alpha_{p_1}\alpha_{p_5}||\alpha_{p_2}\alpha_{p_6}||\alpha_{p_3}\alpha_{p_4}|}{|\alpha_{p_1}\alpha_{p_6}||\alpha_{p_2}\alpha_{p_4}||\alpha_{p_3}\alpha_{p_5}|}=1
\end{equation}
is satisfied;
\item[iii)] there is an involution $f$ on the infinity line which satisfies $$f(H_{p_j})=H_{p_{j+3}}, j \in \Z_6 \quad .$$
\end{enumerate}
\begin{proof}
For simplicity, we write $[\alpha_{p_1},\alpha_{p_2},\alpha_{p_3};\alpha_{p_4},\alpha_{p_5},\alpha_{p_6}]$ instead of
$$\left(|\alpha_{p_1}\alpha_{p_5}||\alpha_{p_2}\alpha_{p_6}||\alpha_{p_3}\alpha_{p_4}|\right) / \left(|\alpha_{p_1}\alpha_{p_6}||\alpha_{p_2}\alpha_{p_4}||\alpha_{p_3}\alpha_{p_5}|\right).$$
There is a canonical isomorphism between $\K^2$ and $(\K^2)^\ast$ which induces an isomorphism between $\mathbb{P}(\K^2)$ and $\mathbb{P}((\K^2)^\ast)$.
Since the vectors $\alpha_p \in(\K^2)^\ast$ are normal to the lines $H_p$, we can study the line arrangement $\A'=\left\{\langle\alpha_p\rangle\right\}_{p=1}^6$,in the dual space $(\K^2)^\ast$,  which trace at infinity $\A'_\infty$ is projectively isomorphic to $\A_\infty$.
\par
\noindent
i) $\Rightarrow$ ii) 
The translate of arrangement $\A'$ (see Figure \ref{fig:Ceva2}) is a Crapo's configuration if and only if there are $c_1,\ldots, c_6\in\K$\footnote{Notice that Figure 2 is obtained from Figure 1 by a $\pi/2$ rotation. This fact guarantees the existence of $c_1, c_2, c_3, c_4, c_5, c_6 $ satisfying the Equation \eqref{eq;Ceva2}. In the case of general fields, the existence of $c_1, c_2, c_3, c_4, c_5, c_6 $ follows from the natural isomorphism between $\K^2$ and $(\K^2)^\ast$.}
 satisfying the equation \begin{equation}\label{eq;Ceva2}
c_4\alpha_{p_4}=c_2\alpha_{p_2}-c_3\alpha_{p_3}, \quad c_5\alpha_{p_5}=c_3\alpha_{p_3}-c_1\alpha_{p_1}, \quad c_6\alpha_{p_6}=c_1\alpha_{p_1}-c_2\alpha_{p_2},
\end{equation}
that is
\begin{align*}
		&[\alpha_{p_1},\alpha_{p_2},\alpha_{p_3};\alpha_{p_4},\alpha_{p_5},\alpha_{p_6}]=
		[c_1\alpha_{p_1},c_2\alpha_{p_2},c_3\alpha_{p_3}; c_4\alpha_{p_4},c_5\alpha_{p_5},c_6\alpha_{p_6}]\\
		=&[c_1\alpha_{p_1},c_2\alpha_{p_2},c_3\alpha_{p_3}; c_2\alpha_{p_2}-c_3\alpha_{p_3}, c_3\alpha_3{p_3}-c_1\alpha_{p_1},c_1\alpha_{p_1}-c_2\alpha_{p_2}] \\
		=&\frac{|\alpha_{p_1}\alpha_{p_3}||\alpha_{p_2}\alpha_{p_1}||\alpha_{p_3}\alpha_{p_2}|}{(-1)^3|\alpha_{p_1}\alpha_{p_2}||\alpha_{p_2}\alpha_{p_3}||\alpha_{p_3}\alpha_{p_1}|} =1.\\
\end{align*}
\noindent
ii) $\Rightarrow$ i) 
Conversely, assume that the equation \eqref{eq;Ceva} is satisfied.
The combinatorics of discriminantal arrangement are invariant under projective transformations for the original arrangements. Thus, we can assume six vectors are $\begin{pmatrix}
      1 & 0 & 1 & \tilde\lambda_{14} & \tilde\lambda_{15} & \tilde\lambda_{16}\\
      0 & 1 & 1 &\tilde\lambda_{24} & \tilde\lambda_{25} & \tilde\lambda_{26}\\
   \end{pmatrix}$.  
But since we also have the flexibility to dilate the six normal vectors in the plane individually by various nonzero constants we can actually assume that 
\begin{equation}\label{eq:coord}
\left(\alpha_{p_1}\alpha_{p_2}\alpha_{p_3}\alpha_{p_4}\alpha_{p_5}\alpha_{p_6}\right)=
   \begin{pmatrix}
      1 & 0 & 1 & \lambda_4 & \lambda_5 & \lambda_6\\
      0 & 1 & 1 & 1 & 1 & 1\\
   \end{pmatrix}
\end{equation}
where $\lambda_4,\lambda_5,\lambda_6\in \K\setminus\{0,1\}$ and $\lambda_4\neq\lambda_5\neq\lambda_6\neq\lambda_4$. Since the Ceva's equation \eqref{eq;Ceva} is satisfied then
\begin{equation}\label{eq:Cevaequiv}
1=[\alpha_{p_1},\alpha_{p_2},\alpha_{p_3};\alpha_{p_4},\alpha_{p_5},\alpha_{p_6}]=\frac{1\cdot (-\lambda_6)\cdot (1-\lambda_4)}{1\cdot (-\lambda_4)\cdot (1-\lambda_5)},
\end{equation} 
that is $(1-\lambda_5)\lambda_4=(1-\lambda_4)\lambda_6$ and if we write $c_1=(1-\lambda_5)\lambda_4,c_2=\lambda_4-1,c_3=\lambda_4$, we get
	\begin{equation*}
		c_1\alpha_{p_1}-c_2\alpha_{p_2}=(1-\lambda_4)\alpha_{p_6}; \quad
		c_2\alpha_{p_2}-c_3\alpha_{p_3}= -\alpha_{p_4};  \quad
		c_3\alpha_{p_3}-c_1\alpha_{p_1}= \lambda_4\alpha_{p_5} \quad .
	\end{equation*}
The equation \eqref{eq;Ceva2} is satisfied and the proof is completed.\\
iii) $\Leftrightarrow$ ii) We are going to prove that, with the choice of coordinates in equation \eqref{eq:coord} then the statement in iii) is equivalent to the equation \eqref{eq:Cevaequiv} which concludes the proof.\\
Let $f$ be a projective involution on the infinity line which satisfies $f(H_{p_j})=H_{p_{j+3}}, j \in \Z_6$, then the following equalities hold:
\begin{alignat*}{5}
[\alpha_{p_1},\alpha_{p_2};\alpha_{p_3},\alpha_{i}]&=&[f(\alpha_{p_1}),f(\alpha_{p_2});f(\alpha_{p_3}),f(\alpha_{i})]&=&[\alpha_{p_4},\alpha_{p_5};\alpha_{p_6},\alpha_{j}]\\
\end{alignat*}
for any couple $(i,j) \in \{(p_4,p_1),(p_5,p_2),(p_6,p_3)\}$, {where the action of $f$ on $(\mathbb{K}^2)^\ast$ is the induced action, that is, $f(\alpha)=\alpha\circ f^{-1}$ for $\alpha\in (\mathbb{K}^2)^\ast$.}
\par
By algebraic computations, with the choice of coordinates in equation \eqref{eq:coord} we get the equations:
$$[\alpha_{p_1},\alpha_{p_2};\alpha_{p_3},\alpha_{p_4}]=\lambda_4,\quad [\alpha_{p_1},\alpha_{p_2};\alpha_{p_3},\alpha_{p_5}]=\lambda_5,\quad [\alpha_{p_1},\alpha_{p_2};\alpha_{p_3},\alpha_{p_6}]=\lambda_6$$  
and, if we set $c=\frac{(\lambda_4-\lambda_6)}{(\lambda_5-\lambda_6)}$, we obtain
$$
[\alpha_{p_4},\alpha_{p_5};\alpha_{p_6},\alpha_{p_1}]=c, \quad [\alpha_{p_4},\alpha_{p_5};\alpha_{p_6},\alpha_{p_2}]=c\frac{\lambda_5}{\lambda_4}, \quad [\alpha_{p_4},\alpha_{p_5}; \alpha_{p_6},\alpha_{p_3}]=c\frac{(\lambda_5-1)}{(\lambda_4-1)} \quad .
$$
Since $f$ is an involution the following equalities hold:
\begin{equation}\label{eq;Inv}
1. \quad c=\lambda_4; \quad 2. \quad c\frac{\lambda_5}{\lambda_4}=\lambda_5; \quad 3. \quad c\frac{(\lambda_5-1)}{(\lambda_4-1)}=\lambda_6 \quad .
\end{equation}
The equation \eqref{eq;Inv} 3. can be written as $(\lambda_4-\lambda_6)(1-\lambda_5)=(\lambda_5-\lambda_6)(1-\lambda_4)\lambda_6$, the equations \eqref{eq;Inv} 1. and \eqref{eq;Inv} 2. are equivalent to \eqref{eq;Inv4} and combining \eqref{eq;Inv} 1. and \eqref{eq;Inv} 3. yields \eqref{eq;Inv5}
\begin{alignat}{2}
(\lambda_4-\lambda_6)-(\lambda_5-\lambda_6)\lambda_4=0\label{eq;Inv4}\\
(1-\lambda_5)\lambda_4=(1-\lambda_4)\lambda_6 \label{eq;Inv5} \quad .
\end{alignat}
If we replace $\lambda_6=\lambda_4(1-\lambda_5)/(1-\lambda_4)$ obtained from the equation \eqref{eq;Inv5} into the left hand side of the equation \eqref{eq;Inv4} we get $$(\lambda_4-\lambda_6)-(\lambda_5-\lambda_6)\lambda_4=\lambda_4(1-\lambda_5)-(1-\lambda_4)\lambda_6=\lambda_4(1-\lambda_5)-\lambda_4(1-\lambda_5)=0 \quad .$$ That is the equation \eqref{eq;Inv5} provides the equation \eqref{eq;Inv4} and hence the equations in \eqref{eq;Inv}. 
On the other hand the equation \eqref{eq;Inv5} is the equation \eqref{eq:Cevaequiv} and the proof follows.
{Conversely, if the condition \eqref{eq:Cevaequiv} is satisfied and we fix the representation matrix of $f$ equals to $\begin{pmatrix}-1&0\\0&\lambda_4\end{pmatrix}^{-1}\begin{pmatrix}\lambda_4&\lambda_5\\1&1\end{pmatrix}^{-1}$, then $f$ is an involution as soon as the $\lambda_i$'s satisfy the above equations.}
\end{proof}
\end{thm}
\noindent
Next corollary has been already proved by Crapo in \cite{Cr85}.
\begin{coro}Let $\A$ be a generic arrangement of $n$ lines in $\K^2$. The number of quadral points of $\B(n,2,\A)$ is either zero or even.
\begin{proof}
Let $\T=\{L_1,\ldots ,L_4\}$ be a $4$-set such that the $\alpha_{p_i}$'s indexed in $\T$ satisfy the Ceva's equation \eqref{eq;Ceva}. Then it is an easy remark that the vectors indexed in the $4$-set $\T'=\{L_1',\ldots,L_4'\},L_i'=\{i_1,\ldots,i_6\}\setminus L_i$ too satisfy the Ceva's equation \eqref{eq;Ceva}.
\end{proof}
\end{coro}
\noindent
The following corollary is a consequence of simple algebraic computations.

\begin{coro}\label{cor;CrCeva}
Let $\T$ be a $4$-set, an arrangement $\A$ admits a $K_\T$-translated which is a Crapo's configuration if and only if the equation 
\begin{equation*}\label{eq;CrCeva}
	[\alpha_{p_2},\alpha_{p_3};\alpha_{p_1},\alpha_{p_4}][\alpha_{p_3},\alpha_{p_1};\alpha_{p_2},\alpha_{p_5}][\alpha_{p_1},\alpha_{p_2};\alpha_{p_3},\alpha_{p_6}]=-1
\end{equation*}
is satisfied.
\end{coro}

\begin{exam}
Let $\A$ be the line arrangement in $\R^2$ defined by the lines normal to the vectors
\begin{alignat*}{1}
A=(\alpha_1 \ldots \alpha_6)=\begin{pmatrix}
1&1&2&0&3&3\\
-1&1&-1&1&-2&1\\ 
\end{pmatrix} \quad .
\end{alignat*}
It is an easy check that the arrangement $\A$ is generic. If we multiply $A$ by the matrix $T=\begin{pmatrix} 3&3\\ -1&-3\\ \end{pmatrix}$ we get
\begin{alignat*}{1}
TA&=
\begin{pmatrix}
3&3\\
-1&-3\\
\end{pmatrix}
\begin{pmatrix}
1&1&2&0&3&3\\
-1&1&-1&1&-2&1\\
\end{pmatrix}\\
&=\begin{pmatrix}
0&6&3&3&3&12\\
2&-4&1&-3&3&-6\\
\end{pmatrix}=\left(2\alpha_4,2\alpha_5,\alpha_6,3\alpha_1,3\alpha_2,6\alpha_3\right) \quad .
\end{alignat*}
That is the map $T$ satisfies the third condition of the Theorem \ref{thm;Ceva} and hence the arrangement $\A$ is non-very generic.
\end{exam}

\begin{exam}\label{exam;Oct}
Let's consider the octahedron on the Riemann sphere $\C\PP^1\cong S^2$; $\{\infty, 0, 1, -1,\allowbreak \sqrt{-1}, -\sqrt{-1}\}$ with the natural action of the octahedral group $(\cong S_4)$ and the matrix of vectors
\begin{equation*}
A=(\alpha_1 \ldots \alpha_6)=
\begin{pmatrix}
      1 & 0 & 1 & -1 & \sqrt{-1} & -\sqrt{-1}\\
      0 & 1 & 1 & 1 & 1 & 1\\
\end{pmatrix}  \quad .
\end{equation*}
Then the six matrices
\begin{alignat*}{6}
&\begin{pmatrix}  1 & 1 \\  1 & -1\\ \end{pmatrix},&&\begin{pmatrix}  -1 & 1 \\  1 & 1\\ \end{pmatrix},&&\begin{pmatrix}  0 & \sqrt{-1} \\  1 & 0\\ \end{pmatrix},\\
&\begin{pmatrix} 0 & -\sqrt{-1} \\  1 & 0\\ \end{pmatrix},&&\begin{pmatrix}  \sqrt{-1} & -1 \\ 1 & - \sqrt{-1} \\ \end{pmatrix},&&\begin{pmatrix}   \sqrt{-1} & 1 \\ -1 & - \sqrt{-1} \\ \end{pmatrix}
\end{alignat*}
act as six involutions with no fixed points on $A$ and if $\A_O$ is the arrangement with the lines normal to the above vectors $\alpha_1,\ldots,\alpha_6$, the discriminantal arrangement $\B(6,2,\A_O)$ has $6\times2=12$ quadral points.
\end{exam}

\subsection{Non-very generic intersections in rank $4$.}Analogously to the previous subsection, we call \textit{quintuple point} a simple intersection $X$ of multiplicity $5$ in rank $4$.
The following proposition is a consequence of the Ceva's Theorem (Theorem \ref{thm;Ceva}) and it provides a condition that yields quintuple points in $\B(n,2,\A)$.

\begin{figure}
\begin{tabular}{cc}
\begin{minipage}{0.5\linewidth}
\centering
	\begin{tikzpicture}[scale=1, xscale=1.6, rotate=50, yscale=0.8 ]
		\draw (-1.2,0) node[at={(2.5,0)}] {$H_{p_0}$}--(2.2,0);
		\draw (-0.2,-1.1) node[at={(-0.4,-1.1)}] {$H_{p_3}$}--(2.2,0.1);
		\draw (-0.2,-0.4) node[at={(1.3,2.6)}] {$H_{p_5}$}--(1.2,2.4);
		\draw (-1.2,0.2) node[at={(0.4,-1.4)}] {$H_{p_1}$}--(0.2,-1.2);
		\draw (-1.2,-0.2) node[at={(1.5,2.2)}] {$H_{p_4}$}--(1.2,2.2);
		\draw (0,-1.2) node[at={(0,-1.5)}] {$H_{p_2}$}--(0,0.3);
		\draw (0.8,2.4) node[at={(0.9,2.7)}] {$H_{p_6}$}--(2.2,-0.4);
	\end{tikzpicture}
	\caption{quintuple point in $\B(n,2,\A)$}\label{fig:Quint1} 
\end{minipage}&
\begin{minipage}{0.5\linewidth}
\centering
	\begin{tikzpicture}[scale=1, xscale=1.6, rotate=50, yscale=0.8 ]
		\draw (-1.2,0) node[at={(2.5,0)}] {$H_{p_0}$}--(2.2,0);
		\draw (-0.2,-1.1) node[at={(-0.4,-1.1)}] {$H_{p_3}$}--(2.2,0.1);
		\draw (-0.2,-0.4) node[at={(1.3,2.6)}] {$H_{p_5}$}--(1.2,2.4);
		\draw (-1.2,0.2) node[at={(0.4,-1.4)}] {$H_{p_1}$}--(0.2,-1.2);
		\draw (-1.2,-0.2) node[at={(1.5,2.2)}] {$H_{p_4}$}--(1.2,2.2);
		\draw (0,-1.2) node[at={(0,-1.5)}] {$H_{p_2}$}--(0,0.3);
		\draw (0.8,2.4) node[at={(0.9,2.7)}] {$H_{p_6}$}--(2.2,-0.4);
		\draw[dashed] (-0.1,-1.3) --node {$H'$}(1.1,2.3);
	\end{tikzpicture}
	\caption{quintuple point add $1$ line}\label{fig:Quint2} 
\end{minipage}
\end{tabular}
\end{figure}

\begin{prop}\label{thm;quint} Let $\T$ be a $5$-set defined by
$$\T=\{\{p_0,p_1,p_4\},\{p_0,p_2,p_5\},\{p_0,p_3,p_6\},\{p_1,p_2,p_3\},\{p_4,p_5,p_6\}\} \quad .$$ 
A simple intersection $X=\bigcap_{L \in \T} D_L$ of hyperplanes in $\B(n,2,\A)$ is a quintuple point if and only if the equation
\begin{equation}\label{eq:quint}
[\alpha_{p_0},\alpha_{p_1};\alpha_{p_2},\alpha_{p_3}]=[\alpha_{p_0},\alpha_{p_4};\alpha_{p_5},\alpha_{p_6}]
\end{equation}
is satisfied. 
\end{prop}
\begin{proof}
Let's assume there is a $K_\T$-translated $\A^t$, $t=t(\T)$, of $\A$ and add the line $H'$, with normal vector $\alpha'$, such that $H^t_{p_1}\cap H^t_{p_2}\cap H^t_{p_3}, H^t_{p_4}\cap H^t_{p_5}\cap H^t_{p_6}\subset H'$.
Then, the new arrangement $\A^t\cup\{H'\}$ contains the $2$ Crapo's configurations $\{H',H^t_{p_0}, H^t_{p_1},H^t_{p_2},H^t_{p_4},H^t_{p_5}\}, \{H',H^t_{p_0}, H^t_{p_1},H^t_{p_3}, H^t_{p_4}, \allowbreak  H^t_{p_6}\}$. 
The Ceva's condition in Theorem \ref{thm;Ceva} yelds
\begin{alignat}{2}
|\alpha_{p_0}\alpha_{p_2}||\alpha_{p_1}\alpha'||\alpha_{p_4}\alpha_{p_5}|/|\alpha_{p_0}\alpha_{p_5}||\alpha_{p_1}\alpha_{p_2}||\alpha_{p_4}\alpha'|&=&1,\label{eq:quint1}\\
|\alpha_{p_0}\alpha_{p_3}||\alpha_{p_1}\alpha'||\alpha_{p_4}\alpha_{p_6}|/|\alpha_{p_0}\alpha_{p_6}||\alpha_{p_1}\alpha_{p_3}||\alpha_{p_4}\alpha'|&=&1.\label{eq:quint2}
\end{alignat}
Combining \eqref{eq:quint1} and \eqref{eq:quint2} yields 
\begin{equation*}\label{eq:quint3}
\frac{|\alpha_{p_4}\alpha'|}{|\alpha_{p_1}\alpha'|}=\frac{|\alpha_{p_0}\alpha_{p_2}||\alpha_{p_4}\alpha_{p_5}|}{|\alpha_{p_0}\alpha_{p_5}||\alpha_{p_1}\alpha_{p_2}|}=\frac{|\alpha_{p_0}\alpha_{p_3}||\alpha_{p_4}\alpha_{p_6}|}{|\alpha_{p_0}\alpha_{p_6}||\alpha_{p_1}\alpha_{p_3}|}
\end{equation*}
from which we obtain $|\alpha_{p_0}\alpha_{p_2}||\alpha_{p_1}\alpha_{p_3}|/|\alpha_{p_0}\alpha_{p_3}||\alpha_{p_1}\alpha_{p_2}|= |\alpha_{p_0}\alpha_{p_5}||\alpha_{p_4}\alpha_{p_6}|/|\alpha_{p_0}\alpha_{p_6}| \allowbreak |\alpha_{p_4}\alpha_{p_5}|$
equivalent to the equation \eqref{eq:quint}. \\
Conversely, suppose that $\A$ satisfies the equation \eqref{eq:quint}.
There is an unique map $f\in PGL(2,\K)$ such that $f(\alpha_{p_1})=\alpha_{p_5}, f(\alpha_{p_5})=\alpha_{p_1}, f(\alpha_{p_2})=\alpha_{p_4}$
 and we choose the homogeneous coordinate such that $(\alpha_{p_1}\alpha_{p_5}\alpha_{p_2}\alpha_{p_4})=\begin{pmatrix}1&0&1&\lambda\\0&1&1&1\end{pmatrix}$.
 Now $f$ is represented by $\begin{pmatrix}0&\lambda \\ 1&0\end{pmatrix}$ therefore $f(\alpha_{p_4})=\alpha_{p_2}$.
 Similarly, by setting $\alpha'=f(\alpha_{p_0})$, the equation \eqref{eq:quint1} holds by $f(\alpha')=\alpha_{p_0}$ and Theorem \ref{thm;Ceva}. As in the first half of the discussion, equation \eqref{eq:quint2} follows from equations \eqref{eq:quint} and \eqref{eq:quint1}. Hence there is a $t$-translated of $\A\cup\{H'\}$ which contains the Crapo configurations $\{H'^t,H^t_{p_0}, H^t_{p_1},H^t_{p_2},H^t_{p_4},H^t_{p_5}\}, \{H'^t,H^t_{p_0}, H^t_{p_1},H^t_{p_3},H^t_{p_4}, \allowbreak H^t_{p_6}\}$  where $H'$ is an hyperplane orthogonal to $\alpha'$.
In particular, $H^t_{p_1}\cap H^t_{p_4},H^t_{p_2}\cap H^t_{p_5},H^t_{p_3}\cap H^t_{p_6}\subset H^t_{p_0}$ and $H^t_{p_1}\cap H^t_{p_2}\cap H^t_{p_3},H^t_{p_4}\cap H^t_{p_5}\cap H^t_{p_6}\neq \emptyset$.
This completes the proof.
\end{proof}
\noindent
Analogously to what happen for the quadral points in $\B(n,2,\A)$ and the non-very generic intersections in $\B(n, 3, \A)$, e.g. the Pappus's Theorem (see \cite{SaSeYa17,SaSeYa19}), there are dependencies between quintuple points too. The following Corollaries are obtained from Proposition \ref{thm;quint} by means of simple algebraic computations. 
\begin{coro} If the intersection lattice of the discriminantal arrangement $\B(n,2,\A)$ contains the two quintuple points associated to the $5$-sets:
\begin{equation*}
\begin{alignedat}{2}
\T_1&=&\{\{p_0,p_1,p_4\},\{p_0,p_2,p_5\},\{p_0,p_3,p_6\},\{p_1,p_2,p_3\},\{p_4,p_5,p_6\}\},\\
\T_2&=&\{\{p_0,p_1,p_5\},\{p_0,p_2,p_6\},\{p_0,p_3,p_4\},\{p_1,p_2,p_3\},\{p_4,p_5,p_6\}\},
\end{alignedat}
\end{equation*}
then it contains the quintuple point associated to the $5$-set:
\begin{equation*}
\T_3=\{\{p_0,p_1,p_6\},\{p_0,p_2,p_4\},\{p_0,p_3,p_5\},\{p_1,p_2,p_3\},\{p_4,p_5,p_6\}\} .
\end{equation*}
\end{coro}

\begin{coro}
If the intersection lattice of the discriminantal arrangement $\B(n,2,\A)$ contains the three quintuple points associated to the $5$-sets:
\begin{alignat*}{2}
\T_1&=&\{\{p_0,p_1,p_4\},\{p_0,p_2,p_5\},\{p_0,p_3,p_6\},\{p_1,p_2,p_3\},\{p_4,p_5,p_6\}\},\\
\T_2&=&\{\{p_0,p_1,p_4\},\{p_0,p_2,p_5\},\{p_0,p_3,p_6\},\{p_1,p_5,p_3\},\{p_4,p_2,p_6\}\},\\
\T_3&=&\{\{p_0,p_1,p_4\},\{p_0,p_2,p_5\},\{p_0,p_3,p_6\},\{p_1,p_2,p_6\},\{p_4,p_5,p_3\}\}.
\end{alignat*}
Ithen it contains the quintuple point associated to the $5$-set:
\begin{equation*}
\T_4=\{\{p_0,p_1,p_4\},\{p_0,p_2,p_5\},\{p_0,p_3,p_6\},\{p_1,p_5,p_6\},\{p_4,p_2,p_3\}\}.
\end{equation*}
\end{coro}
\subsection{Arrangements from regular polygons} For simplicity let's call \textit{quintuple point} a simple intersection of multiplicity $5$ and rank $4$. Given a line  arrangement $\A$, let's denote by $m_4(\A)$ and $m_5(\A)$ the numbers, respectively, of the quadral and quintuple points of the discriminantal arrangement $\B(n,2,\A)$.\par
Notice that from the characterization of non-very generic intersections by means of projective transformations shows that highly symmetric arrangements have many quadral or quintuple points.
If we denote by $R_n$ the central line arrangement defined by the $n$ lines parallel to the edges of the regular polygon with $2n$ sides, then the following Proposition holds. 

\begin{prop}
Let $R_n$ be a real central arrangement consisting of the $n$ lines normal to the vectors $\alpha_p=(\cos \frac{p\pi}{n}, \sin \frac{p\pi}{n})$ for $p\in[n]$, then for $n\geq 6$ we have
\begin{equation}\label{eq:quad}
\begin{split}
m_4(R_n) \geq (n+2)\binom{n/2}{3}+n\binom{n/2-1}{3} \quad \mbox{if $n$ is even} \\
m_4(R_n) \geq 2n\binom{(n-1)/2}{3}\quad  \mbox{if $n$ is odd}
\end{split}
\end{equation}
while for $n\geq 7$ we have 
\begin{equation}\label{eq:quint}
\begin{split}
m_5(R_n) \geq 4n\binom{n/2-1}{3} \quad \mbox{if $n$ is even} \\
m_5(R_n) \geq 4n\binom{(n-1)/2}{3} \quad \mbox{if $n$ is odd}.
\end{split}
\end{equation}
\end{prop}
\begin{proof}
Set $[A_n]=\{[\alpha_1],\ldots,[\alpha_n]\}\subset\PP(\R^2)$ and 
let $\rfl(\alpha,\cdot)$ be the reflection map $v\mapsto v-\left(2v\cdot \alpha/\alpha\cdot \alpha\right)\alpha$ where $v\cdot \alpha$ is the inner product of $v$ and $\alpha$.
The reflections $\rfl(\alpha_{p_1},\cdot)$ and $\rfl(\alpha_{p_1}+\alpha_{p_2},\cdot)$ are involutions acting on $[A_n]$. Indeed $\rfl(\alpha,\cdot)$ is an element of $PGL(2,\R)$ since $[\rfl(\alpha_{p_1},\alpha)],[\rfl(\alpha_{p_1}+\alpha_{p_2},\alpha)]\in [A_n]$ for any $[\alpha_{p_1}],[\alpha_{p_2}],[\alpha]\in [A_n]$ and $\rfl(v,\cdot)^2=\mathrm{id}$ for any $[v]\in\PP(\R^2)$. In particular we have that
\begin{equation*}
[\rfl(\alpha_p,\alpha_{p+q})]=[\alpha_{p-q}],\,[\rfl(\alpha_{p-1}+\alpha_p,\alpha_{p+q})]=[\alpha_{p-q-1}]
\end{equation*}
where we regard the index set $[n]$ as the cyclic group $\Z_n$.\\
Since the fixed points of $\langle[\rfl(\alpha_{p_1},\cdot)]\rangle$ are, respectively, $[\alpha_p]$ if $n$ is odd and $[\alpha_{p+n/2}]$ if $n$ is even, the orbit decompositions of $[A_n]$ by the action of $\langle[\rfl(\alpha_{p_1},\cdot)]\rangle$ are
$$\{\{\alpha_p\},\{\alpha_{p-1},\alpha_{p+1}\},\ldots,\allowbreak \{\alpha_{p-(n-1)/2},\allowbreak \alpha_{p+(n-1)/2}\}\} \quad \mbox{if n is odd}$$
and
$$\{\{\alpha_p\},\{\alpha_{p-1},\alpha_{p+1}\}, \ldots,\{\alpha_{p-n/2+1},\alpha_{p+n/2-1}\},\{\alpha_{p+n/2}\}\} \quad \mbox{if n is even} \quad.$$
The orbit decomposition of $[A_n], n$ even, by the action of $\langle\rfl(\alpha_{p-1}+\alpha_p,\cdot)\rangle$ is
$$\{\{\alpha_{p-1},\alpha_{p}\},\{\alpha_{p-2},\allowbreak \alpha_{p+1}\},\ldots,\allowbreak \{\alpha_{p-n/2},\alpha_{p+n/2-1}\}\} \quad.$$
We can prove the Proposition case by case as follows.
\begin{enumerate}
\item[\ref{eq:quad}.1] For each reflection $\rfl(\alpha_p,\cdot)$ and $\rfl(\alpha_{p-1}+\alpha_p,\cdot)$, $p=1,\ldots,n/2$, by Theorem \ref{thm;Ceva} any choice of three elements, respectively in $\{\{\alpha_{p-1},\alpha_{p+1}\},\ldots,\allowbreak \{\alpha_{p-n/2+1},\alpha_{p+n/2-1}\}\}$ and in $\{\{\alpha_{p-1},\alpha_{p}\},\{\alpha_{p-2},\alpha_{p+1}\},\ldots,\allowbreak \{\alpha_{p-n/2},\alpha_{p+n/2-1}\}\}$ gives rise to two quadral points. Moreover the matrix $\begin{pmatrix}0&-1\\ 1&0\end{pmatrix}$ provides an involution on $[A_n]$ which maps $\alpha_p$ into $\alpha_{p+n/2}$ and hence any choice of three elements in $\{\{\alpha_{1}, \alpha_{1+n/2}\},\allowbreak \{\alpha_{2}, \alpha_{2+n/2}\},\ldots, \{\alpha_{n/2}, \alpha_{n}\}\}$ gives rise to two more quadral points.
Thus we obtain exactly $\frac{n}{2}\binom{n/2}{3}\times2+\frac{n}{2}\binom{n/2-1}{3}\times2+\binom{n/2}{3}\times2$ quadral points.
\item[\ref{eq:quad}.2] Analogously to [\ref{eq:quad}.1], for each reflection $\rfl(\alpha_p,\cdot), p=1,\ldots, n$, any choice of three elements in $\{\{\alpha_{p-1},\alpha_{p+1}\},\ldots,\{\alpha_{p-(n-1)/2},\alpha_{p+(n-1)/2}\}\}$ gives rise to two quadral points and, summing up, we obtain exactly $2n\binom{(n-1)/2}{3}$ quadral points.
\item[\ref{eq:quint}.1] By Proposition \ref{thm;quint}, for each reflection $\rfl(\alpha_p,\cdot),p=1,\ldots,n/2$, choosing three elements from $\{\{\alpha_{p-1},\alpha_{p+1}\},\ldots,\allowbreak \{\alpha_{p-\frac{n}{2}+1},\alpha_{p+\frac{n}{2}-1}\}\}$ and one from $\{\alpha_p,\alpha_{p+n/2}\}$ gives rise to four quintuple points. Thus we obtain the number $\frac{n}{2}\binom{n/2-1}{3}\times2\times4$.
\item[\ref{eq:quint}.2] Analogously to [\ref{eq:quint}.1], for each reflection $\rfl(\alpha_p,\cdot), p=1,\ldots,n$, any choice of three elements from $\{\{\alpha_{p-1},\alpha_{p+1}\},\ldots,\allowbreak \{\alpha_{p-(n-1)/2},\alpha_{p+(n-1)/2}\}\}$ gives rise to four quintuple points.
Thus we obtain the number $\frac{n}{2}\binom{(n-1)/2}{3}\times2\times4$.
\end{enumerate}
\end{proof}



\section{The non-very generic intersections in $\B(6,2,\A)$ and $\B(6,3,\A)$}\label{Sec:Per}
\noindent
In the previous Section we proved that the only non-very generic intersections in $\B(6,2,\A)$ are pairs of quadral points which correspond to involutions in $PGL(2,\K)$. It follows that there is a natural correspondence between the non-very generic intersections in $\B(6,2,\A)$ and the permutations of the form $(p_1p_2)(p_3p_4)(p_5p_6)$.\\
Moreover, in \cite{SaSeYa17} authors show that there is a correspondence between the permutation $(p_1p_2)(p_3p_4)(p_5p_6)$ and the non-very generic intersection in $\B(6,3,\A)$ associated to the $3$-set $\T=\{\{p_1,p_2,p_3,p_4\},\{p_3,p_4,p_5,p_6\},\{p_1,p_2,p_5,p_6\}\}$.\\ 
As a consequence of the above remarks, it is possible to classify the non-very generic intersections in $\B(6,2,\A)$ and $\B(6,3,\A)$ by means of the permutation group.
\subsection{Permutation group on six points}\label{SSec:Per}
\begin{figure}
\begin{tabular}{lr}
\begin{minipage}{0.65\linewidth}
\centering
\begin{footnotesize}
\begin{tikzpicture}[scale=1.5, rotate=90]
	\node[inner sep=0.2em, fill=black!100, circle] (a) at (2,0){};
	\node[inner sep=0.2em, fill=black!100, circle] (b) at (1,1.73){};
	\node[inner sep=0.2em, fill=black!100, circle] (c) at (-1,1.73){};
	\node[inner sep=0.2em, fill=black!100, circle] (d) at (-2,0){};
	\node[inner sep=0.2em, fill=black!100, circle] (e) at (-1,-1.73){};
	\node[inner sep=0.2em, fill=black!100, circle] (f) at (1,-1.73){};
	\node at (2,0.5){$\phi(1)$};
	\node at (1.3,1.8){$\phi(2)$};
	\node at (-1.3,1.8){$\phi(3)$};
	\node at (-2,-0.5){$\phi(4)$};
	\node at (-1.3,-1.8){$\phi(5)$};
	\node at (1.3,-1.8){$\phi(6)$};
	
	\draw (b) to node[auto=left] {$\phi((12))=(15)(26)(34)$} (a);
	\draw (c) to node[auto=left, rotate=60] {$(14)(23)(56)$} (a);
	\draw (d) to node[auto=left, rotate=-90, pos=0.15] {$(13)(25)(46)$} (a);
	\draw (a) to node[auto=left, rotate=-60, pos=0.1] {$(12)(36)(45)$} (e);
	\draw (a) to node[auto=left] {$(16)(24)(35)$} (f);
	
	\draw (c) to node[auto=left] {$(12)(35)(46)$} (b);
	\draw (d) to node[auto=left, rotate=-60, pos=0.1] {$(16)(23)(45)$} (b);
	\draw (b) to node[auto=left, rotate=-30, pos=0.15] {$(13)(24)(56)$} (e);
	\draw (b) to node[auto=left, pos=0.3] {$(14)(25)(36)$} (f);
	
	\draw (d) to node[auto=left] {$(15)(24)(36)$} (c);
	\draw (e) to node[auto=left, pos=0.3] {$(16)(25)(34)$} (c);
	\draw (c) to node[auto=left, rotate=30, pos=0.45] {$(13)(26)(45)$} (f);
	
	\draw (e) to node[auto=left] {$(14)(26)(35)$} (d);
	\draw (f) to node[auto=left, rotate=60] {$(12)(34)(56)$} (d);
	
	\draw (f) to node[auto=left] {$(15)(23)(46)$} (e);
\end{tikzpicture}
\end{footnotesize}
\caption{Complete graph $\phi(K_6)$}\label{fig:phiK6}
\end{minipage}&

\begin{minipage}{0.3\linewidth}
\centering
\begin{tikzpicture}[scale=0.6, rotate=90]
	\node[inner sep=0.2em, draw, circle] (a) at (2,0){$1$};
	\node[inner sep=0.2em, draw, circle] (b) at (1,1.73){$2$};
	\node[inner sep=0.2em, draw, circle] (c) at (-1,1.73){$3$};
	\node[inner sep=0.2em, draw, circle] (d) at (-2,0){$4$};
	\node[inner sep=0.2em, draw, circle] (e) at (-1,-1.73){$5$};
	\node[inner sep=0.2em, draw, circle] (f) at (1,-1.73){$6$};
	
	\draw[dash dot] (b) to (a);
	\draw[dashed] (c) to (a);
	\draw[double distance=0.8pt] (d) to (a);
	\draw (a) to (e);
	\draw[line width=0.8pt, dotted] (a) to (f);
	
	\draw[double distance=0.8pt] (c) to (b);
	\draw[line width=0.8pt, dotted] (d) to (b);
	\draw[dashed] (b) to (e);
	\draw (b) to (f);
	
	\draw (d) to (c);
	\draw[line width=0.8pt, dotted] (e) to (c);
	\draw[dash dot] (c) to (f);
	
	\draw[dash dot] (e) to (d);
	\draw[dashed] (f) to (d);
	
	\draw[double distance=0.8pt] (f) to (e);
\end{tikzpicture}
\begin{tikzpicture}
\draw (0,2)--(1,2)node[right]{$(15)(26)(34)$};
\draw[double distance=0.8pt] (0,1.5)--(1,1.5)node[right]{$(14)(25)(36)$};
\draw[dashed] (0,1)--(1,1)node[right]{$(13)(25)(46)$};
\draw[dash dot] (0,0.5)--(1,0.5)node[right]{$(12)(36)(45)$};
\draw[line width=0.8pt, dotted] (0,0)--(1,0)node[right]{$(16)(24)(35)$};
\end{tikzpicture}
\caption{$1$-factorization on $\phi(1)$}\label{fig:Parti}
\end{minipage}
\end{tabular}
\end{figure}
\noindent
Let $S_6$ be the symmetric group of degree $6$ and $\phi$ be the outer automorphism defined by:
\begin{alignat*}{7}
(12)&\mapsto&\phi((12))=(15)(26)(34),&\quad&
(23)&\mapsto&\phi((23))=(12)(35)(46),\\
(34)&\mapsto&\phi((34))=(15)(24)(36),&\quad&
(45)&\mapsto&\phi((45))=(14)(26)(35),\\
(56)&\mapsto&\phi((56))=(15)(23)(46).
\end{alignat*}
Notice that $\phi$ is unique up to inner automorphisms of $S_6$. \\
\noindent
If $K_6$ is the complete graph on $6$ vertices, it is possible to define the new graph $\phi(K_6)$, depicted in Figure \ref{fig:phiK6}, having vertices $\phi(i)$'s and edges $\phi(i)\phi(j)=\phi((ij))$. As a consequence of the fact that $\phi$ is an automorphisms, we get that  the subgroups $\langle E\rangle, \langle E'\rangle$ of the symmetric group $S_6$ associated to any two subsets $E,E'$ of the edge set of $\phi(K_6)$, are equal if and only if the two partitions of the vertices determined by the connected components of $E$ and $E'$ are the same. That is the partitions of the vertices of $\phi(K_6)$ determine the subgroups of $S_6$.\\
There are exactly the following eleven types of partitions of six:
\begin{equation*}
1^6, 1^4 2^1, 1^3 3^1, 1^2 2^2, 1^2 4^1, 1^1 2^1 3^1, 2^3, 1^1 5^1, 2^1 4^1, 3^2, 6^1
\end{equation*}
which define, up to connected components, eleven types of subgraphs of $\phi(K_6)$. In particular a partition $\V_{\nu}$ of type $\nu=i_1^{n_1} i_2^{n_2}\cdots i_l^{n_l}$ will correspond to a subgraph $\phi(K_6)[\V_{\nu}]$ of $\phi(K_6)$ having as connected components exactly $n_j$ copies of $K_{i_j},j=1,\ldots,l$.
The number $m(\nu)$ of the edges in $\phi(K_6)[\V_{\nu}]$ equals $\sum_{1\leq j\leq l} n_j\binom{i_j}{2}$, that is:
\begin{equation}
\begin{alignedat}{8}\label{eq:mformula}
m(1^6)&=0, &\quad m(1^4 2^1)&=1, &\quad m(1^3 3^1)&=3, &\quad m(1^2 2^2)&=2, \\
m(1^2 4^1)&=6, &\quad m(1^1 2^1 3^1)&=4, &\quad m(2^3)&=3, &\quad m(1^1 5^1)&=10, \\
m(2^1 4^1)&=7, &\quad m(3^2)&=6, &\quad m(6^1)&=15.
\end{alignedat}
\end{equation}
Since the edges of $\phi(K_6)$ correspond to elements $\sigma \in S_6$ which are the product of exactly three transpositions and each $\sigma$, in turns, corresponds to a non very generic intersections in $\B(6,2,\A)$ and $\B(6,3,\A)$, we can now count and, partially, classify, the non very generic intersections in $\B(6,2,\A)$ and $\B(6,3,\A)$.

\subsection{The arrangement $\B(6,2,\A)$}
\noindent
Let $\A$ be a generic arrangement of $6$ lines in $\K^2$ and denote by $\mathcal{Q}_4(\B(6,2,\A))$ the set of quadral points of $\B(6,2,\A)$ and by $m_4(\A)$ its cardinality. With the above notations, the discussion in Section \ref{SSec:Per} provides a proof of the following theorem.

\begin{thm}\label{thm:typeB62}
There is one and only one type $\nu$ partition $\V_{\nu}$ of the vertex set of $\phi(K_6)$ such that the edge set of $\phi(K_6)[\V_{\nu}]$ corresponds to a fixed $4$-set associated to $\mathcal{Q}_4(\B(6,2,\A))$.
In particular $m_4(\A)=2m(\nu)$.
\end{thm}

\noindent The above Theorem \ref{thm:typeB62} allows to classify the non-very generic arrangements of $6$ lines in $\K^2$. For example, the arrangement $\A_O$ in the Example \ref{exam;Oct} corresponds to type $1^2 4^1$. Notice, however, that not all partitions give rise to a non-very generic arrangement. Indeed the existence of an arrangement corresponding to a given partition of $[6]={1,\cdots,6}$ depends on the coefficient field $\K$.\\ 
The following proposition easily follows from the properties of the projective transformation group.

\begin{prop}
There are not line arrangements corresponding to a type $6^1$ partition of the vertex set of $\phi(K_6)$. In particular, $m_4(\A)\leq 2m(1^1 5^1)=20$ for any arrangement $\A$.
\end{prop}
\begin{proof}
If an arrangement $\A$ corresponding to a partition of type $6^1$ exists, then $S_6$ acts faithfully on $\A_\infty$.
In particular, there is a permutation exchanging only two points but fixing the other four points in the $PGL(2,\K)$.
But a projective transformation that fixes three points in the $PGL(2,\K)$ is the identity map and this is is an absurd.
\end{proof}
\noindent
The above proposition provides a sharp upper bound for $m_4(\A)$. For example, $m_4(\A)$ is exactly $20$ when $\A$ is the unique arrangement of type $1^1 5^1$ over the five elements finite field $\mathbb{F}_5$.
Its normals are the colmuns of $\begin{pmatrix}
    1&0&1&2&3&4\\0&1&1&1&1&1
\end{pmatrix}$.

\subsection{The arrangement $\B(6,3,\A)$}
\noindent
In this subsection, $\A$ is a generic arrangement of $6$ planes in $\K^3$, $\mathcal{Q}_3(\B(6,3,\A))$ the set of non very generic intersections of $\B(6,3,\A)$ and $m_3(\A)$ its cardinality. Analogously to the planar case, the following Theorem holds.
\begin{thm}\label{thm:ParmB63}
There is one and only one type $\nu$ partition $\V_{\nu}$ of the vertex set of $\phi(K_6)$ such that the edge set of $\phi(K_6)[\V_{\nu}]$ corresponds to a fixed $3$-set associated to $\mathcal{Q}_3(\B(6,3,\A))$.
In particular $m_3(\A)=m(\nu)$.
\end{thm}
\noindent
A consequence of our approach is the extension to any field of the alternative proof of the Pappus's Hexagon Theorem over the complex field provided by Sawada, Yamagata, and the second author in \cite{SaSeYa17} . Analogously to the following one, their proof is based on the combinatorics of the discriminantal arrangement.
\begin{thm}[Pappus]
Let $\T_1,\T_2$ be two disjoint $3$-sets associated to two non-very generic intersections in $\B(6,3,\A)$ and let  $\sigma_1$ and $\sigma_2$ be edges of $\phi(K_6)$ corresponding to $\T_1$ and $\T_2$ respectively. Then the $3$-set $\T_3$ corresponding to the edge $\sigma_3=\sigma_1\sigma_2\sigma_1$ gives rise to a third non-very generic intersection (see Figure \ref{fig:Pappus}).
\begin{figure}
\begin{tikzpicture}
\draw[dashed](-3,1)--(2.5,1)node[pos=-0.6]{$\sigma_1:\{\{p_1,p_2\},\{p_3,p_4\},\{p_5,p_6\}\}$};
\draw[dashed](-3,0)--(2.5,0)node[pos=-0.6]{$\sigma_3:\{\{p_1,p_4\},\{p_2,p_5\},\{p_3,p_6\}\}$};
\draw[dashed](-3,-1)--(2.5,-1)node[pos=-0.6]{$\sigma_2:\{\{p_1,p_6\},\{p_2,p_3\},\{p_4,p_5\}\}$};
\draw (-2.4,1.2)--(2.4,-1.2);
\draw (-2.2,1.2)--(0.2,-1.2);
\draw (2.2,-1.2)--(-0.2,1.2);
\draw (-2.2,-1.2)--(0.2,1.2);
\draw (2.2,1.2)--(-0.2,-1.2);
\draw (2.4,1.2)--(-2.4,-1.2);
\node at (-2.9,1.3){$H_{p_2,\infty}$};
\node at (-2.1,1.5){$H_{p_1,\infty}$};
\node at (-0.5,1.5){$H_{p_3,\infty}$};
\node at (0.5,1.5){$H_{p_4,\infty}$};
\node at (2.0,1.5){$H_{p_6,\infty}$};
\node at (2.9,1.3){$H_{p_5,\infty}$};
\end{tikzpicture}
\caption{Pappus' configuration}\label{fig:Pappus}
\end{figure}
\end{thm}
\subsection{A complete classification of $\B(6,3,\A)$ over a commutative field of characteristic $0$}\label{Sec:Class}

\begin{thm}\label{thm:Class}
Let $\K$ be a subfield of the complex field $\C$ and let $\A$ be an arrangement of $6$ planes in $\K^3$.
The minimal fields in which the type $\nu$ arrangement appears are provided in the following table:
\begin{equation*}
	\begin{array}{l||c|c|c|c|c|c|c|c|c|c|c}
	\nu&1^6&1^4 2^1&1^3 3^1&1^2 2^2&1^2 4^1&1^1 2^1 3^1&2^3&1^1 5^1&2^1 4^1&3^2&6^1 \\ \hline
	\K&\Q&\Q&\Q&\Q&\Q&\Q&\star&\Q(\sqrt{5})&\star&\Q(\sqrt{-3})&\star
	\end{array}
\end{equation*}
Here $\star$ means that the type $\nu$ arrangement does not exist over a characteristic $0$ commutative field.
\end{thm}
\begin{proof}
The Falk's example (Example 3.2. \cite{Fa94}) shows that the type $1^1 2^1 3^1$ arrangement appears over $\Q$.
Since the condition corresponding to each edge of $\phi(K_6)$ is a solution set of some algebraic equation, the one corresponding to the refinement of $1^1 2^1 3^1$ can be realized on $\Q$ by a proper perturbation.
Therefore the types $1^6,1^42^1,1^33^1,1^22^2,1^12^13^1$ appear in $\Q$.\\
In the rest of the proof we set the vectors normal to the planes in $\A$ as follows: 
\begin{equation*}
\left(\alpha_1\alpha_2\alpha_3\alpha_4\alpha_5\alpha_6\right)=
\begin{pmatrix}
1&0&0&1&w&y\\
0&1&0&1&x&z\\
0&0&1&1&1&1\\
\end{pmatrix}.
\end{equation*}
The condition that $\A$ is generic implies that:
\begin{equation}\label{eq:condition}
w,x,y,z\not\in\{0,1\},\quad w-x,w-y,x-z,y-z,wz-xy,w-x-y+z-wz+xy\neq 0.
\end{equation}
Let $\alpha_p\times\alpha_{p'}$ be the cross product of $\alpha_p$ and $\alpha_{p'}$, we will classify each case by means of the Lemma \ref{thm:ParmB63} and the equation
\begin{equation}\label{eq:cross}
\det(\alpha_{p_1}\times\alpha_{p_2},\alpha_{p_3}\times\alpha_{p_4},\alpha_{p_5}\times\alpha_{p_6})=0
\end{equation}
 which is equivalent to the fact that the intersection $\bigcap_{L \in \T} D_{L}, \T=\{\{p_1,p_2,p_3,p_4\},\allowbreak \{p_3,p_4, p_5,p_6\}, \{p_1,p_2,p_5,p_6\}\}$ is a non-very generic intersection in $\B(3,6,\A)$ (see \cite{SaSeYa17}). \\
For simplicity, we write $\det((p_1p_2)(p_3p_4)(p_5p_6))$ instead of $\det(\alpha_{p_1}\times\alpha_{p_2},\alpha_{p_3}\times\alpha_{p_4},\alpha_{p_5}\times\alpha_{p_6})$.\\
\noindent
\textbf{Type $1^2 4^1,\,1^1 5^1$ :}
Let $\V_{1^2 4^1}=\{\{\phi(1),\phi(2),\phi(3),\phi(4)\},\allowbreak \{\phi(5)\},\{\phi(6)\}\}$ be a partition of type $1^2 4^1$. With this choice the edges of $\phi(K_6)[1^2 4^1]$ are determined by $(15)(26)(34),(12)(35)(46)$, and $(15)(24)(36)$.
Then the following equations are satisfied.
\begin{equation*}
\left\{\begin{alignedat}{5}
0=\det((15)(26)(34))&=&\det\Matrix{0}{-1}{x}{1}{0}{-y}{-1}{1}{0}&=&x-y,\\
0=\det((12)(35)(46))&=&\det\Matrix{0}{0}{1}{-x}{w}{0}{1-z}{y-1}{z-y}&=&-xy+x-w+wz,\\
0=\det((15)(24)(36))&=&\det\Matrix{0}{-1}{x}{1}{0}{-1}{-z}{y}{0}&=&-xy+z.\\
\end{alignedat}\right.
\end{equation*}
Organizing the above equations, we get 
\begin{equation}\label{eq:1241}
w(1-x^2)=x-x^2,\,y=x,\,z=x^2
\end{equation}
which has a solution over the rationals.\\
Analogous computations for the type $1^1 5^1$ and the choice of the partition $\V_{1^1 5^1}=\{\{\phi(1),\phi(2),\phi(3),\phi(4),\phi(5)\},\{\phi(6)\}\}$ yield the equation
$$
x^2+x-1
$$
which has solution over the field $\Q(\sqrt{5})$.\\
\noindent
\textbf{Type $3^2$ :}
Analogously to the previous case, if we choose $\V_{3^2}=\{\{\phi(1),\phi(2),\phi(3)\},\allowbreak \{\phi(4),\phi(5),\phi(6)\}\}$, then the edges of $\phi(K_6)[3^2]$ are determined by $(15)(26)(34),\allowbreak(12)(35)(46),(14)(24)(35)$, and $(15)(23)(46)$. By algebraic computations we get the equations
\begin{equation*}
w=x^2,\,y=z,\,z=-x^2+2x,\,x^4-2x^3+2x^2-x=0.
\end{equation*}
The last equation yields $x(x-1)(x^2-x+1)=0$ which, by $x\neq 0,1$ has solution over the field $\Q(\sqrt{-3})$.\\ 
\noindent
\textbf{Type $2^3,\, 2^1 4^1,\, 6^1$ :}
Let $\V_{2^3}=\{\{\phi(1),\phi(2)\},\allowbreak\{\phi(3),\phi(4)\}, \{\phi(5),\phi(6)\}\}$ be a partition of type $2^3$. Then the edges of $\phi(K_6)[2^3]$ are $(15)(26)(34),(15)(24)(36)$, and $(15)(23)(46)$ which yield
\begin{equation*}
\left\{\begin{alignedat}{5}
0=\det((15)(26)(34))&&=&&x-y,\\
0=\det((15)(24)(36))&&=&&xy+z,\\
0=\det((15)(23)(46))&&=&&xy-x+z-y,
\end{alignedat}\right.
\end{equation*}
that is
\begin{equation}\label{eq:nonreal}
x=y,\,x=-y,\,z=-xy \quad .
\end{equation}
The equations in (\ref{eq:nonreal}) are satisfied if and only if $x=y=0$ which are not admissible values by \eqref{eq:condition}.
Moreover, since the types $2^1 4^1$ and $6^1$ have $2^3$ as a substructure, the equations in \eqref{eq:nonreal} have to be satisfied in those cases too. Hence we can conclude that there are not arrangements of type $2^3,\, 2^1 4^1$, and $6^1$ in any characteristic $0$ commutative field.
\end{proof}

\begin{rema}
The induced graph by the partition type $1^6$ has no edges, so it corresponds to the very generic arrangement.\\
The type $3^2$ corresponds to the Hesse configuration mentioned by Sawada, Yamagata, and the second author (\cf Theorem $6.5.$ \cite{SaSeYa19}) and it does not appear over the real field.
So this classification includes their results.\\
Notice that, as a consequence of \eqref{eq:mformula}, two discriminantal arrangements associated to arrangements of $6$ planes in $\C^3$ can have the same number of non-very generic points, but different combinatorics. On the other hand, this is a classification over characteristic $0$ commutative fields
thus we cannot exclude the possibility that there are more types of general position arrangements of $6$ planes in a $3$-dimensional space over a skew field such that the combinatorics of the discriminantal arrangements associated to them are not isomorphic to the ones we described here.
\end{rema}

\begin{rema}
In the case of type $6^1$, the arrangement must satisfy $5$ equations but the essential parameters of any arrangement of $6$ planes in a $3$-dimensional space are $4$ so the type $6^1$ arrangement does not appear over most fields.\\
However the condition \eqref{eq:nonreal} has a solution when the characteristic of the ground field is $2$ so, for example, over the $4$ elements finite field $\mathbb{F}_4=\mathbb{F}_2[\omega]/(\omega^2+\omega+1)$ if we choose the normals as 
\begin{equation*}
(\alpha_1\alpha_2\alpha_3\alpha_4\alpha_5\alpha_6)=
\begin{pmatrix}
1&0&0&1&\omega&\omega^2\\
0&1&0&1&\omega^2&\omega\\
0&0&1&1&1&1
\end{pmatrix}
\end{equation*}
we get a type $6^1$ arrangement $\A_{6^1}$ with exactly $15$, i.e. the maximum, non-very generic intersections in $\B(6,3,\A)$. Indeed in this case $w=z=\omega,\,x=y=\omega^2$ and all the equations in the proof of Theorem \ref{thm:Class} are satisfied.
If we consider $\B(6,3,\A_{6^1})$ in terms of matroid theory, it includes the uniform matroid $U_{2,4}$ and rank $2$ projective plane $\PP(\mathbb{F}_2^3)$ as minors.
Thus $\B(6,3,\A_{6^1})$ is only representable over characteristic $2$ fields $\mathbb{F}$ with $\abs{\mathbb{F}}>4$ if the field $\mathbb{F}$ is commutative.
\end{rema}


\section{The minimum intersection lattice of $\B(6,3,\A)$ over the real field}\label{Sec:Dodeca}
\begin{figure}
	\begin{center}
	\includegraphics[keepaspectratio, scale=0.2] {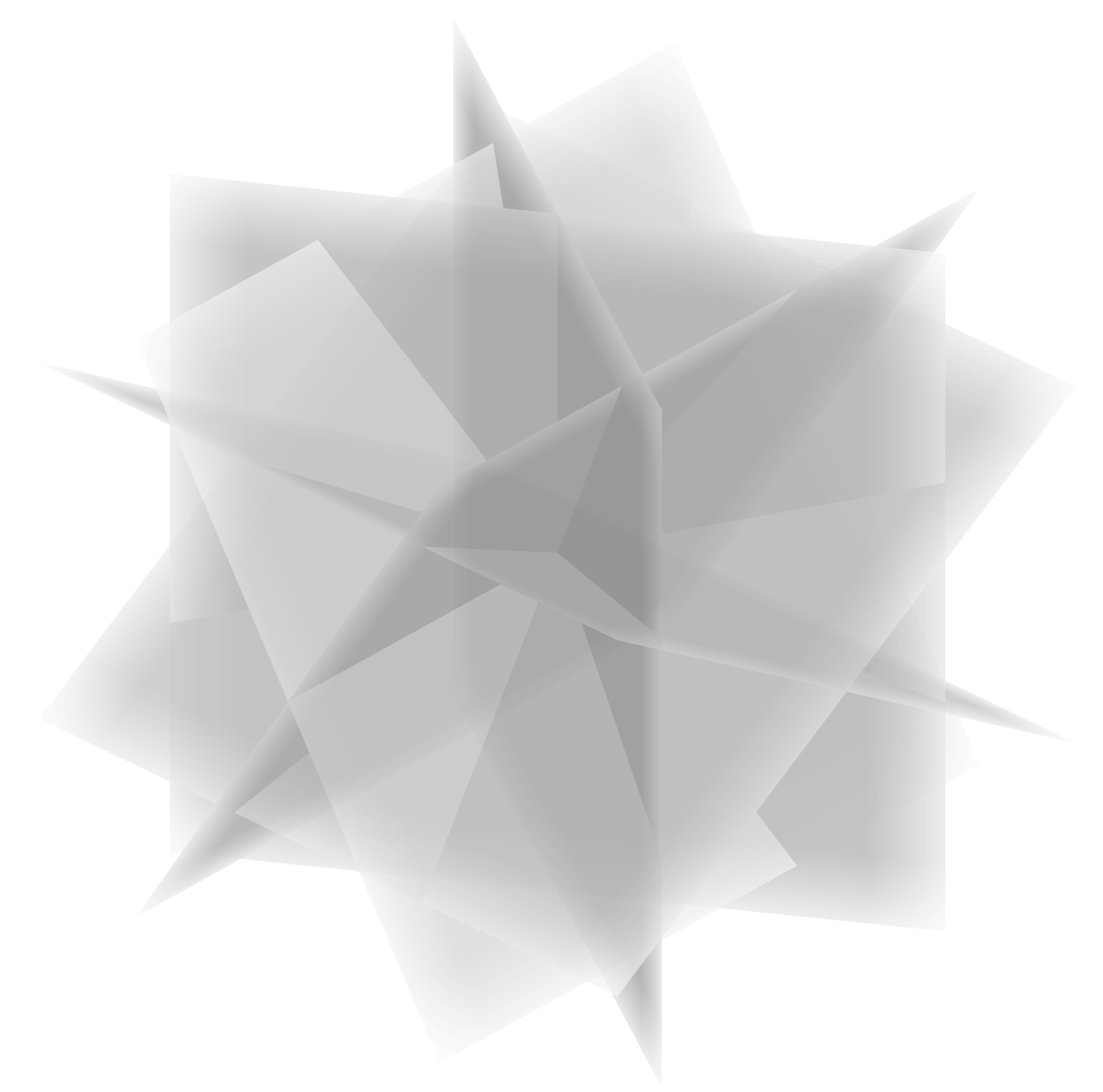}
	\caption{Dodecahedral arrangement}\label{fig;D}
	\end{center}
\end{figure}
\noindent
In Section \ref{Sec:Class} we proved that the maximum number of non-very generic intersections in $\B(6,3,\A)$ over the real (or complex) field is $10$. In this Section we provide an example of such an arrangement: the \textit{dodecahedral arrangement} $\A_{D}=\left\{H_1,H_2,H_3,H_4,H_5,H_6\right\}$ in $\R^3$ defined by the $6$ planes parallel to the faces of the dodecahedron. The normals to the hyperplanes of $\A_D$ are given by
$$\left(\alpha_1 \alpha_2 \alpha_3 \alpha_4 \alpha_5 \alpha_6\right)=
 \begin{pmatrix}
      1 & 1 & 0 & 0 & \tau& -\tau\\
      0 & 0 & \tau & -\tau & 1 & 1\\
      \tau & -\tau & 1 & 1 & 0 & 0\\
   \end{pmatrix}$$
where $\tau$ is the golden ratio $\tau=\frac{1+\sqrt{5}}{2}$. The following proposition holds.

\begin{prop}
The discriminantal arrangement associated to the dodecahedral arrangement has $10$ non-very generic intersections.
\begin{proof}
It is an easy computation to show that the $10$ $3$-sets $\T_i,i=1,\ldots,10$ listed in the table \eqref{tab;allgood3partitions} are associated to hyperplanes $D_L \in \B(6,3,\A_D), L \in \T_i$ which normal vectors satisfy the dependency conditions listed in the equation (\ref{eq:dip}). That is the $10$ intersections $\bigcap_{L \in \T_i} D_L$ are non-very generic.
\begin{equation*}\label{tab;allgood3partitions}
\begin{array}{rccc|rccc}
		&L_1\cap L_2& L_1\cap L_3 & L_2\cap L_3 & 
		&L_1\cap L_2& L_1\cap L_3 & L_2\cap L_3\\ \hline
	\T_1:&12&35&46&
	\T_2:&12&36&45\\
	\T_3:&13&26&45&
	\T_4:&13&24&56\\
	\T_5:&14&23&56&
	\T_6:&14&25&36\\
	\T_7:&15&23&46&
	\T_8:&15&26&34\\
	\T_9:&16&24&35&
	\T_{10}:&16&25&34
\end{array}
\end{equation*}
\begin{equation}\label{eq:dip}
\begin{aligned}
&\T_1:&\alpha_{1235}-\alpha_{1246}-\alpha_{3546}&=&0,&\quad&
&\T_2:&\alpha_{1236}-\alpha_{1245}-\alpha_{3456}&=&0,\\
&\T_3:& \alpha_{1236}+\alpha_{1345}+\alpha_{2456}&=&0,&\quad&
&\T_4:&\alpha_{1356}-\alpha_{1234}-\alpha_{2456}&=&0,\\
&\T_5:&\alpha_{1456}-\alpha_{1234}-\alpha_{2356}&=&0,&\quad&
&\T_6:&\alpha_{1346}-\alpha_{1245}-\alpha_{2356}&=&0,\\
&\T_7:&\alpha_{1456}-\alpha_{1235}-\alpha_{2346}&=&0,&\quad&
&\T_8:&\alpha_{2346}-\alpha_{1256}-\alpha_{1345}&=&0,\\
&\T_9:&\alpha_{1246}-\alpha_{1356}-\alpha_{2345}&=&0,&\quad&
&\T_{10}:&\alpha_{1346}-\alpha_{1256}-\alpha_{2345}&=&0.\\
\end{aligned}
\end{equation}
\end{proof}
\end{prop}
\noindent
This is also a counterexample to the Remark $6.6$ in \cite{SaSeYa19}.

\bibliographystyle{amsplain-nodash}
\bibliography{refarence}
\end{document}